\documentclass[11pt]{article}
\usepackage{amsmath,amsthm,amscd,amsfonts,amssymb,enumerate}
\usepackage{graphicx}		
\usepackage{color}

\usepackage{tikz}
\usepackage{caption}
\usepackage{pgf,tikz}
\usetikzlibrary{arrows}
\tikzstyle{vertex}=[circle, draw, inner sep=0pt, minimum size=6pt]

\usepackage{amsmath,amsthm,amscd,amsfonts,amssymb,enumerate}
\usepackage{graphicx}		
\usepackage{color}
\usepackage{mathrsfs}
\newtheorem{prelem}{{\bf Theorem}}

\newtheorem{theorem}{Theorem}

\newtheorem{lemma}[theorem]{Lemma}

\theoremstyle{definition}
\newtheorem{definition}[theorem]{Definition}
\title{Exact double domination in subdivision, Mycielskian  and middle graphs}

\date{}
\author{
{ Ali Behtoei\thanks{\tt Corresponding author, a.behtoei@sci.ikiu.ac.ir}}\\
Rozhin Zarifian \thanks{\tt rozhinazarifian@gmail.com} \\
{\it \small Department of Mathematics, Faculty of Science,} \\ {\it \small Imam Khomeini International University,} \\
{ \it \small Qazvin, Iran, PO Box: 34148 - 96818.}  
}
\begin{document}
	\maketitle
% ----------------------------------------------------------------
\begin{abstract}
%Sierpi\'{n}ski and Sierpi\'{n}ski gasket graphs have many applications and are studied in diverse areas including fractal theory, dynamic systems, topology, chemistry, psychology, probability, and computer science. 
An exact (or efficient) doubly dominating set for a graph $G=(V,E)$ is a subset $D$ of vertices such that each vertex of $G$ is dominated by exactly two vertices of $D$.
In this paper we study and determine the existence of exact doubly dominating sets for the famous structures subdivision, middle and Mycielskian. Also, we determine  the existence of exact doubly dominating sets  for their complements.
\vspace{3mm}\\
{\bf Keywords:} Exact double domination, subdivision, middle, Mycielskian. \\
{\bf MSC 2010}: 05C69, 05C76.
\end{abstract}

\section{Introduction}

Let $G=(V,E)$ be a finite and simple  graph with  vertex set $V=V(G)$ and edge set $E=E(G)$. 
The set $N_G(u)$ denotes the (open) neighborhood of $u\in V(G)$, which means the set of all adjacent vertices  to $u$ in $G$, the closed neighborhood of $u$ is $N_G[u]=N_G(u)\cup\{u\}$ and the degree of  $u$ is $\deg_G(u)=|N_G(u)|$. An isolated vertex is a vertex of degree zero.
A clique $C$ in $G$ is a subset of vertices of $G$ such that every two distinct vertices in $C$ are adjacent and hence, the induced subgraph of $G$ on it is a complete graph. %The clique number of $G$, $\omega(G)$, is the maximum size among all cliques of $G$.
A matching in $G$ is a set of pairwise non-incident edges of $E$ and,
a perfect matching is a matching in which every vertex of the graph is incident to exactly one edge of the matching.
For a subset of vertices $X\subseteq V$, $G[X]$ denotes the subgraph induced by $X$. Note that $X$ is an independent set in $G$ If and only if $G[X]$ contains no edge. The complement graph of $G$ is denoted by $\overline{G}$.
In recent years much attention drawn to the domination theory which has different applications in diverse areas and is an interesting branch in graph theory.
Each vertex of a graph is said to dominate every vertex in its closed neighborhood. 
A subset $S$ of $V(G)$ is a dominating set for $G$ if each vertex in $V(G) \setminus S$ is adjacent to at least one vertex in $S$. The domination number of $G,$ denoted by $\gamma(G),$ is the minimum size of a dominating set of $G$. 
In \cite{Fink} Fink and Jacobson generalized the concept of dominating sets. 
Let $k$ be a positive integer. A subset $D$ of vertices in $G$ is a $k$-dominating set if each vertex in $V(G)\setminus D$ is adjacent to at least $k$ vertices in $D$. The $k$-domination number $\gamma_k(G)$ is the minimum cardinality of a $k$-dominating set of $G$. Hence, for $k=1$, $1$-dominating sets are the classical dominating sets. 
A vertex subset $D$ is a perfect $k$-dominating set if each vertex $v$ of $G$, not in $D$, is
adjacent to exactly $k$ vertices of $D$. The perfect k-domination problem is NP-complete for general graphs.  Note that every nontrivial graph has a perfect $k$-dominating set,
since the entire vertex set is such a set and there are graphs whose only perfect
$k$-dominating set is their entire vertex set (consider the stars $K_{1,t}$ for $1<k<t$). 
%Recently, the concept of domination is  expanded to other parameters of domination.
A paired-dominating set  is a dominating set of vertices whose induced subgraph has a perfect matching.
A set $S\subseteq V$ is a double dominating set for $G$ if each vertex in $V$ is dominated by at least two vertices in $S$, note that this concept is different from $2$-domination. 
%The minimum cardinality of a perfect k-dominating set of G is the perfect k-domination number $\gamma_{kp}(G)$. 
Cockayne et al. in \cite{PerfectDom} called a vertex subset $D$  of $V$ to be a perfect dominating set of $G$ if every vertex in $V(G)\setminus D$ is adjacent to exactly one vertex of $D$, see also \cite{Fellows}. 
Note that sets that are both perfect dominating and independent are called perfect codes by
Biggs in \cite{Biggs} or efficient dominating sets by Bange, Barkauskas and Slater in \cite{BangeSiam}.
Analogously to  perfect or  efficient domination, %introduced \cite{BangeSiam},
 Harary and Haynes in \cite{Harary2000} defined an efficient doubly dominating set as a subset $D$ of vertices such that each vertex of $G$ is dominated by exactly two vertices of $D$, i.e $|N_G[v]\cap D|=2$ for each $v\in V$,  see also  \cite{Harary1996}. 
Chellali, Khelladi and Maffray in \cite{Chellali2005} prefer to use the phrase exact doubly dominating set for this concept and they show that the complexity of the problem of deciding whether a graph admits an exact doubly dominating set is NP-Complete.  
Note that not all graphs  admit an exact doubly dominating set (for example consider the 4-cycle $C_4$ or the stars). 
In \cite{Chellali2005} a constructive characterization of those trees that admit an exact
doubly dominating set is provided, and they establish a necessary and sufficient
condition for the existence of an exact doubly dominating set in a connected cubic graph. 
 
\begin{theorem} \label{Matching}
\cite{Chellali2005} The vertex set of every exact doubly dominating set induces a matching. 
Moreover, if $G$ has an exact doubly dominating set, then all such sets have the same size.
\end{theorem}

\begin{theorem} \label{PnCn}
\cite{Chellali2005} \\ a) A path $P_n$ has an exact doubly dominating set if and only
if $n\equiv 2~ (mod ~3)$. If this holds, then the size of any such set is ${2(n + 1)\over 3}$. \\
b) A cycle $C_n$ has an exact doubly dominating set if and only
if $n\equiv 0~ (mod ~3)$. If this holds, then the size of any such set is ${2n \over 3}$. 
\end{theorem}

%%%%%%%%%%%%%%%%%%%%%%%%%%%%%%%%%%%%%%%%%
%%%%%%%%%%%%%%%%%%%%%%%%%%%%%%%%%%%%%%%%%
%%%%%%%%%%%%%%%%%%%%%%%%%%%%%%%%%%%%%%%%%
%%%%%%%%%%%%%%%%%%%%%%%%%%%%%%%%%%%%%%%%%

\section{Subdivision, Mycielskian and middle }

The subdivision operation  is an operation that replaces any edge of a graph by a path of order at least two.  %and the resulting graph of this operation is a  subdivided graph. 
 If each edge is replaced by a path of order three (i.e., 1-subdividing each edge of $G$), then the subdivision graph is denoted by $S(G)$, see $S(P_5)$ in Figure \ref{subdiv-myciels}. 
Domination number and identifying code number of the
subdivision of some famous families of graphs are investigated and determined in \cite{S-Ahmadi}. 
Some upper and lower bounds for the mixed metric dimension of $S(G)$ is provided in \cite{Yero1}.
The minimum number of edges that must be subdivided in order to increase the total $k$-rainbow domination number of a graph is considered in \cite{REF4}. 
Also,  $2-$rainbow domination number of the subdivision graph of some famous families of graphs is determined in \cite{Salkhori}.
Recall that an edge contraction is an operation that removes an edge while simultaneously merging the two (end) vertices that it previously joined.
In the following, we first investigate the existence of exact doubly dominating sets in the subdivision graph $S(G)$ (see Theorem \ref{S(G)}).

%%%%%%%%%%%%%%
\begin{center}
\begin{tikzpicture}
[inner sep=0.5mm, place/.style={circle,draw=black,fill=red,thick}]
%%%%%%% 

\node[circle,draw=black,fill=white,thick] (v1) at (0,0) [label=below:$v_1$] {};
\node[circle,draw=black,fill=white,thick] (v2) at (1,0) [label=below:$v_2$] {} edge [-,thick,blue](v1) ; 
\node[circle,draw=black,fill=white,thick] (v3) at (2,0) [label=below:$v_3$] {} edge [-,thick,blue](v2) ;
\node[circle,draw=black,fill=white,thick] (v4) at (3,0) [label=below:$v_4$] {} edge [-,thick,blue](v3) ;
\node[circle,draw=black,fill=white,thick] (v5) at (4,0) [label=below:$v_5$] {} edge [-,thick,blue](v4) ;
\node (dots) at (2,-1) [label=center: $P_5$]{};

\node[circle,draw=black,fill=white,thick] (v1) at (5,0) [label=below:$v_1$] {};
\node[circle,draw=black,fill=white,thick] (v2) at (6,0) [label=below:$v_2$] {} edge [-,thick,blue](v1) ; 
\node[circle,draw=black,fill=white,thick] (v3) at (7,0) [label=below:$v_3$] {} edge [-,thick,blue](v2) ;
\node[circle,draw=black,fill=white,thick] (v4) at (8,0) [label=below:$v_4$] {} edge [-,thick,blue](v3) ;
\node[circle,draw=black,fill=white,thick] (v5) at (9,0) [label=below:$v_5$] {} edge [-,thick,blue](v4) ;
\node[circle,draw=black,fill=orange,thick] (e1) at (5.5,0) [label=above:$z_{_{12}}$] {};
\node[circle,draw=black,fill=orange,thick] (e2) at (6.5,0) [label=above:$z_{_{23}}$] {};
\node[circle,draw=black,fill=orange,thick] (e3) at (7.5,0) [label=above:$z_{_{34}}$] {};
\node[circle,draw=black,fill=orange,thick] (e4) at (8.5,0) [label=above:$z_{_{45}}$] {};
\node (dots) at (7,-1) [label=center: $S(P_5)$]{};

\node[circle,draw=black,fill=white,thick] (v1) at (10,0) [label=below:$v_1$] {};
\node[circle,draw=black,fill=white,thick] (v2) at (11,0) [label=below:$v_2$] {} edge [-,thick,blue](v1) ; 
\node[circle,draw=black,fill=white,thick] (v3) at (12,0) [label=below:$v_3$] {} edge [-,thick,blue](v2) ;
\node[circle,draw=black,fill=white,thick] (v4) at (13,0) [label=below:$v_4$] {} edge [-,thick,blue](v3) ;
\node[circle,draw=black,fill=white,thick] (v5) at (14,0) [label=below:$v_5$] {} edge [-,thick,blue](v4) ;
\node[circle,draw=black,fill=white,thick] (u1) at (10,1) [label=below:$u_1$] {}edge [-,thick,blue](v2) ;
\node[circle,draw=black,fill=white,thick] (u2) at (11,1) [label=below:$u_2$] {} edge [-,thick,blue](v1) edge [-,thick,blue](v3); 
\node[circle,draw=black,fill=white,thick] (u3) at (12,1) [label=below:$u_3$] {} edge [-,thick,blue](v2) edge [-,thick,blue](v4);
\node[circle,draw=black,fill=white,thick] (u4) at (13,1) [label=below:$u_4$] {} edge [-,thick,blue](v3) edge [-,thick,blue](v5);
\node[circle,draw=black,fill=white,thick] (u5) at (14,1) [label=below:$u_5$] {} edge [-,thick,blue](v4) ;
\node[circle,draw=black,fill=white,thick] (w) at (12,1.5) [label=above:$w$] {} edge [-,thick,blue](u1) edge [-,thick,blue](u2) edge [-,thick,blue](u3) edge [-,thick,blue](u4) edge [-,thick,blue](u5);
\node (dots) at (12,-1) [label=center: $\mu(P_5)$]{};
%%%%%%%
\end{tikzpicture}
\captionof{figure}{    \label{subdiv-myciels}}
\end{center}
%%%%%%%%%%%%%%

\begin{lemma} \label{NotV(G)}
Let $G$ be an $n$-vertex graph and assume that there exists an exact doubly dominating set $D$ for its subdivision $S(G)$. Then,  $\big{|}D\cap V(G)\big{|} \leq n-1$ and the equality holds if and only if  $G\in \{P_3,C_3\}$.
\end{lemma}
\begin{proof}
Assume that  $V(G)=\{v_1,v_2,...,v_n\}$ and  $V(S(G))=V(G)\cup Z$ in which 
$$Z=\big\{z_{ij}:~v_iv_j\in E(G) \big\},~N_{S(G)}(z_{ij})=\{v_i,v_j\}.$$
Note that $V(G)$ induces an independent set in $S(G)$ and for each $v_i\in V(G)$ we have $$|N_{S(G)}[v_i]\cap V(G)|=|\{v_i\}|=1\neq 2.$$ 
Hence, $D\neq V(G)$.
If $V(G)\subsetneqq D$, then for each $z_{_{ij}} \in Z$ we obviously  have $$|N_{S(G)}[z_{_{ij}}]\cap D|=|\{v_i,v_j,z_{ij}\}|=3,$$ which is a contradiction.
Thus,  $\big{|}D\cap V(G)\big{|} \leq n-1$. \\
It is easy to check that the bound is attained for $G\in \{P_3,C_3\}$. 
Now let $G$ be a graph for which  $\big{|}D\cap V(G)\big{|} = n-1$. Assume that $V(G)\setminus D=\{v_i\}$.
Since $D$ is an exact doubly dominating set and $v_i\notin D$, there exist exactly two indices $j$ and $j'$ such that $\{ z_{ij}, z_{ij'}\}\subseteq D \cap Z$. Thus, $\{v_iv_j,v_iv_{j'}\}\subseteq E(G)$. If there exists an index $j''$ such that $v_iv_{j''}\in E(G)$, 
then $z_{ij''}\notin D$ and this implies that $N_{S(G)}[z_{ij''}]\cap D=\{v_{j''}\}$, which is a contradiction. Thus, $v_i$ is a vertex of degree two in $G$.
Since $V(G)\setminus \{v_i\}\subseteq D$, for each $z_{rs}\in Z\setminus\{z_{ij},z_{ij'}\}$ we have $\{v_r,v_s\}\subseteq N_{S(G)}[z_{rs}]\cap D$.
Hence, $z_{rs}\notin D$. 
This means that for each vertex $v_r\notin\{v_i,v_j,v_{j'}\}$ (if any one exists!) we have $$N_{S(G)}[v_r]\cap D =\{v_r\},$$ which is a contradiction.
Thus, $V(G)=\{v_i,v_j,v_{j'}\}$ and hence, $G\in \{P_3,C_3\}$.
\end{proof}

\begin{definition}
Let $\Gamma$ be the family of all finite and simple graphs containing a perfect matching. 
Suppose that $H\in \Gamma$ and  $M\subseteq E(H)$ is a perfect matching for $H$. 
Define the new graph $S_M(H)$ to be the graph obtained from $H$ by 1-subdividing each edge of $M$.
Note that $|V(H)|=2|M|$ and $$|V(S_M(H))|=|V(H)|+|M|=3~\! |M|={3 ~\!|V(H)|\over 2},$$ 
which specially implies that $|V(S_M(H))|$ is a multiple of $3$.
Also, let 
$$\mathcal{G}=\big\{ S_M(H):~ H\in \Gamma ~ \mbox{and $M$ is a perfect matching for} ~ H \big\}.$$ 
\end{definition}

\begin{theorem} \label{S(G)}
Let $G$ be a  graph. Then, there exists an exact doubly dominating set $D$ for $S(G)$ if and only if $G\in \mathcal{G}$.
%$3\big{|} n$ and there exists a set $\Omega\subseteq \{x:~x\in V(G), ~\deg_G(x)=2\}$ such that $|\Omega|= {n\over 3}$ and $N_G[\Omega]=V(G)$. 
In this case we have $|D|={4 ~\!\! |V(G)| \over 3}$.
\end{theorem}
\begin{proof}
At first, assume that $G\in \mathcal{G}$. Hence, $G=S_M(H)$ for some graph $H$ and some perfect matching $M$ of $H$. 
%Let $n=|V(G)|$ and note that $3\big{|} n$.
Let $\Omega$ be the set of (new) vertices of $G$ which are obtained by 1-subdividing each edge of matching $M$ in $H$. 
Hence, $V(G)=V(H)\cup \Omega$. 
Assume that  $V(G)=\{v_1,v_2,...,v_n\}$ and  $V(S(G))=V(G)\cup Z$ in which 
$$Z=\big\{z_{ij}:~v_iv_j\in E(G)\big\},~~~~N_{S(G)}(z_{ij})=\{v_i,v_j\}.$$
Note that for each $v_i\in \Omega$ we have $\deg_G(v_i)=2$.
Since $M$ is a perfect matching for $H$, for each pair of distinct vertices $v_i,v_{i'}$ in $\Omega$ we have $N_G(v_i)\cap N_G(v_{i'})=\emptyset$ and $$\bigcup_{v_i\in \Omega} N_G(v_i)=V(G)\setminus \Omega . $$
Specially, $\Omega$ is a dominating set for $G$ and $|V(G)|=3|\Omega|$ (or equivalently, $|\Omega|={n\over 3}$).
Now consider the subdivision graph $S(G)$ and let 
$$D=\big( V(G)\setminus \Omega \big) \bigcup \bigg( \bigcup_{v_i\in \Omega} N_{S(G)}(v_i) \bigg)   .$$
For each $v_i\in \Omega$ we have 
$$\big| D\cap N_{S(G)}[v_i] \big|=|N_{S(G)}(v_i)|=\deg_G(v_i)=2.$$
For each $v_j\in V(G)\setminus \Omega$, there exists unique $v_i\in \Omega$ such that $v_jv_i\in E(G)$ and hence, 
$$\big| D\cap N_{S(G)}[v_j] \big|=\big|\{v_j,z_{ij}\}\big|=2.$$
For each $z_{rs}\in Z$, if $\{v_r,v_s\}\cap \Omega=\emptyset$, then  we have
$$\big| D\cap N_{S(G)}[z_{rs}] \big|=\big|\{v_r,v_s\}\big|=2,$$
and otherwise, we have  $\big|\{v_r,v_s\}\cap\Omega\big|=1$ and  $z_{rs} \in D$, which implies that 
$$\big| D\cap N_{S(G)}[z_{rs}] \big|=2.$$
Therefore, $D$ is an exact doubly dominating set for $S(G)$ and, we obviously have
$$|D|=(n-{n\over 3})+{n\over 3}\times 2={4n\over 3},$$
as desired.\\
%%%%%%%%%%%
%%%%%%%%%%%
%%%%%%%%%%%
Now assume that $G$ is an $n$-vertex graph such that  there exists an exact doubly dominating set $D$ for $S(G)$. We want to show that $G\in \mathcal{G}$.
As before, assume that  $V(G)=\{v_1,v_2,...,v_n\}$ and  $V(S(G))=V(G)\cup Z$ in which 
$$Z=\big\{z_{ij}:~v_iv_j\in E(G)\big\},~~~~N_{S(G)}(z_{ij})=\{v_i,v_j\}.$$
Let $\Omega=V(G)\setminus D$. 
By Lemma \ref{NotV(G)}, we have $V(G)\setminus D\neq \emptyset$ and hence, $\Omega \neq \emptyset$.
Since $D$ is an exact doubly dominating set for $S(G)$, for each $z_{ij}\in V(S(G))$ we have 
$$2=\big{|}D\cap N_{S(G)}[z_{ij}]\big{|}=\big{|}D\cap \{v_i,z_{ij},v_j\}\big{|}.$$
Thus, if $v_i\in \Omega$, then $\{v_j,z_{ij}\}\subseteq D$ for each $v_j\in N_G(v_i)$. Specially,  $\big| D\cap N_{S(G)}[v_i] \big|=2$ implies that $\deg_G(v_i)=2$.
By Theorem \ref{Matching}, $D$ induces a matching in $S(G)$. Thus, if $v_{j'}\in D$, then there exists unique $v_{i'}\in N_G(v_{j'})$ such that $z_{i'j'}\in D \cap N_{S(G)}(v_{j'})$. Hence,  $v_{i'}\notin D$ which means that $v_{i'}\in \Omega$. Now the previous statement implies that $\deg_G(v_{i'})=2$. 
Therefore, for each pair of different vertices $v_i,v_{i'}$ in $\Omega$ we have $N_G(v_i)\cap N_G(v_{i'})=\emptyset$ and $$\bigcup_{v_i\in \Omega} N_G(v_i)=V(G)\setminus \Omega . $$
Specially, we can see that 
$$|V(G)|=|\Omega|+|V(G)\setminus \Omega|=|\Omega|+2|\Omega|=3|\Omega|,$$
and
$$ |D|=|D\cap V(G)|+|D\cap Z|=\big( |V(G)|-|\Omega| \big) +2|\Omega|=4 |\Omega|=4{|V(G)|\over 3}~.$$
Now for each $v_i\in \Omega$, contract exactly one of two incident edges of $v_i$ in $G$ and, let $H$ be the resulting graph. Note that $|V(H)|=2|\Omega|$ and $|E(H)|=|E(G)|-|\Omega|$. Also, let $M\subseteq E(H)$ be the set of remaining incident edges of $v_i$'s (the elements of $\Omega$). 
Note that by the previous facts and this construction, $M$ is a perfect matching for $H$ and $S_M(H)=G$. This means that $G\in \mathcal{G}$, as desired.
\end{proof}

%%%%%%%%%%%%%%%%%%%%%%%%%%%%%%%%%%%%
%%%%%%%%%%%%%%%%%%%%%%%%%%%%%%%%%%%%
\begin{theorem} \label{BarS(G)}
Let $G$ be a  graph. Then, there exists an exact doubly dominating set $D$ for $\overline{S(G)}$ if and only if $G$ contains at least two isolated vertices and $D$ consists of two isolated vertices of $G$.
\end{theorem}
\begin{proof}
Obviously, each isolated vertex of $G$ is adjacent to all other vertices of $\overline{S(G)}$ in $\overline{S(G)}$.
Hence, if $G$ contains at least two isolated vertices, then each set $D$ consisting of two isolated vertices of $G$ is an exact doubly dominating set  for $\overline{S(G)}$. Note that if we use the previous notations, then $V(G)$ (and similarly, $Z$) induces an independent set in $S(G)$ and hence, induces a clique in its complement  $\overline{S(G)}$.\\
Now assume that $G$ is a graph such that there exists an exact doubly dominating set $D$ for $\overline{S(G)}$. 
Since for each $v_i\in V(G)$ we  have $\big| D\cap N_{_{\overline{S(G)}}}[v_i] \big|=2$ and $V(G)$ induces a clique in $\overline{S(G)}$, we must have $|D\cap V(G)|\leq 2$. Similarly, we have  $|D\cap Z|\leq 2$.\\
If $|D\cap V(G)| =0$, then by using Theorem \ref{Matching}, we obtain  $|D\cap Z|= 2$ and hence, $D=\{z_{ij},z_{rs}\}$ for two distinct vertices  $z_{ij},z_{rs}\in Z$.
Note that $z_{ij}\in Z$ implies that $v_iv_j\in E(G)$ and hence, $v_i$ is not adjacent to $z_{ij}$ in  $\overline{S(G)}$.
Thus, $\big| D\cap N_{_{\overline{S(G)}}}[v_i] \big| \leq 1$, which is a contradiction.\\
If $|D\cap V(G)| =1$, then $D\cap V(G)=\{v_i\}$ for some $v_i\in V(G)$ and Theorem \ref{Matching} implies that  there exists unique vertex  $z_{rs}\in D\cap Z$ which is adjacent to $v_i$ in  $\overline{S(G)}$ and hence, $v_i\notin \{v_r,v_s\}$.
Since $Z$ induces a clique, again Theorem \ref{Matching} implies that $D=\{v_i,z_{rs}\}$. Thus, $\big| D\cap N_{_{\overline{S(G)}}}[v_r] \big| =|\{v_i\}|= 1$, a contradiction. \\
Therefore, we must have $|D\cap V(G)| = 2$. Assume that $D\cap V(G)=\{v_i,v_j\}$. 
Note that $v_i$ and $v_j$ are adjacent in $\overline{S(G)}$.
By Theorem \ref{Matching}, for each $z_{rs}\in D$ we must have $\{v_r,v_s\}\cap \{v_i,v_j\}\neq\emptyset$ because $z_{rs}$ must be non-adjacent to $v_i$ and $v_j$ in  $\overline{S(G)}$.
If there exists $z_{is}\in D$ for some $s\neq j$, then 
$$\big| D\cap N_{_{\overline{S(G)}}}[v_j] \big| =|\{v_i,v_j,z_{is}\}|= 3,$$ a contradiction.
Similarly, $z_{js} \notin D$ for each $s\neq i$.
Thus, $D\subseteq \{v_i,v_j,z_{ij}\}$ and now Theorem \ref{Matching} implies that $D=\{v_i,v_j\}$.
Specially, each member of $Z$ must be adjacent to both $v_i$ and $v_j$ in $\overline{S(G)}$, which means that each member of $Z$ is non-adjacent to both $v_i$ and $v_j$ in $S(G)$. Hence, $v_i$ and $v_j$ are two isolated vertices in $G$ and the proof is complete.
\end{proof}
%%%%%%%%%%%%%%%%%%%%%%%%%%%%%%%%%%%%%%%%%
%%%%%%%%%%%%%%%%%%%%%%%%%%%%%%%%%%%%%%%%%
%%%%%%%%%%%%%%%%%%%%%%%%%%%%%%%%%%%%%%%%%
%%%%%%%%%%%%%%%%%%%%%%%%%%%%%%%%%%%%%%%%%
%The Mycielski construction is a method for turning a triangle-free graph with chromatic number $k$ into a larger triangle-free graph with chromatic number $k+1$. 
Let $G$ be graph with vertex set $V(G)=\{v_1,v_2,...,v_n\}$. The Mycielski graph $\mu(G)$ of $G$ is a graph of order $2n+1$ with the vertex set $V(G)\cup\{w,u_1,u_2,...,u_n\}$ and with the edge set $E(G)\cup \{v_iu_j:~v_iv_j\in E(G)\}\cup \{wu_i:~1\leq i\leq n\}$, see $\mu(P_5)$ in Figure \ref{subdiv-myciels}. 

\begin{theorem}
Let $G$ be an arbitrary graph. Then,  the Mycielskian graph $\mu(G)$ has not any exact doubly dominating set.
\end{theorem}
\begin{proof}
Suppose on the contrary that $D\subseteq V(\mu(G))$ is an exact doubly dominating set for $\mu(G)$. Note that 
$N_{\mu(G)}(w)=\{u_1,u_2,...,u_n\}$. 
If $w\in D$, then Theorem \ref{Matching} implies that there exists unique vertex $u_i$, $1\leq i\leq n$, such that $u_i\in D$.
Hence, $D\cap \{u_1,u_2,...,u_n\}=\{u_i\}$.
Since $\{w,u_i\} \subseteq D\cap N_{\mu(G)}[u_i]$ and $\big| D\cap N_{\mu(G)}[u_i]\big| =2$, for each $v_j\in N_G(v_i)\subseteq N_{\mu(G)}(u_i)$ we have $v_j\notin D$. Thus, $N_{\mu(G)}[v_i] \cap D \subseteq \{v_i\}$ and hence, $|N_{\mu(G)}[v_i] \cap D|\neq 2$,  which is a contradiction.
This contradiction implies that $w\notin D$. 
Since $|N_{\mu(G)}[w] \cap D|=2$, we must have $D\cap \{u_1,u_2,...,u_n\}=\{u_i,u_j\}$ for some $i\neq j$.
By Theorem \ref{Matching}, $u_i$ is adjacent to a (unique) vertex of $D$. Hence, there exists unique vertex $v_r\in N_G(v_i) =\big(N_{\mu(G)}(u_i)\cap V(G)\big)$ such that $v_r\in D$.
If $v_i\in D$, then $\{v_i,v_r,u_i\}\subseteq \big( N_{\mu(G)}[v_r]\cap D\big)$, which is a contradiction. Thus, $v_i\notin D$ and this implies that $N_{\mu(G)}[v_i]\cap D=\{v_r\}$. This also leads to another contradiction and hence, the proof is complete.
\end{proof}

%Given a (simple) graph $G$, its complement $\overline{G}$ is a graph with the same vertex set $V(G)$ such that for each pair of vertices $u$ and $v$, $uv\in E(\overline{G})$ if and only if $uv\notin E(\overline{G})$.

\begin{theorem}
The complement of the Mycielskian graph $\mu(G)$ (i.e., $\overline{\mu(G)}$) has an exact doubly dominating set $D$ if and only if
% $D\subseteq V(G)$ and 
$D$ consists of exactly two isolated vertices of $G$. 
\end{theorem}
\begin{proof}
Let $D\subseteq V(\overline{\mu(G)})=V(\mu(G))$ be an exact doubly dominating set for $\overline{\mu(G)}$.
Note that the set $\{u_1,u_2,...,u_n\}$ induces a clique in $\overline{\mu(G)}$ and we have $N_{_{\overline{\mu(G)}} }(w)=\{v_1,v_2,...,v_n\}$. \\
If $w\in D$, then Theorem \ref{Matching} implies that there exists unique vertex $v_i$, $1\leq i\leq n$, such that $D\cap \{v_1,v_2,...,v_n\}=\{v_i\}$. Since $v_iu_i\in E(\overline{\mu(G)})$ and $wu_i\notin E(\overline{\mu(G)})$, the condition  $|N_{_{\overline{\mu(G)}}}[u_i] \cap D|=2$ implies that $|D\cap \{u_1,u_2,...,u_n\}|=1$. This means that $|D|=3$ which contradicts Theorem  \ref{Matching}. 
Thus, $w\notin D$. Since $|N_{_{\overline{\mu(G)}}}[w]\cap D|=2$, there exist $i\neq j$ such that $D\cap \{v_1,v_2,...,v_n\}=\{v_i,v_j\}$.
Since the set $\{u_1,u_2,...,u_n\}$ induces a clique in $\overline{\mu(G)}$, $v_iu_i\in E(\overline{\mu(G)})$ and $|N_{_{\overline{\mu(G)}}}[u_i]\cap D|=2$, we must have $|D\cap \{u_1,u_2,...,u_n\}|\leq 1$. Thus, 
$$|D|=\big| D\cap \{v_1,v_2,...,v_n\} \big| +\big| D\cap \{u_1,u_2,...,u_n\} \big| \leq 2+1.$$
Now Theorem  \ref{Matching} implies that $|D| =2$. Hence, $D=\{v_i,v_j\}$ and again Theorem \ref{Matching} implies that $v_iv_j\in E(\overline{\mu(G)})$.
This means that $v_iv_j\notin E(\mu(G))$. Also, for each $v_r\in V(G)\setminus \{v_i,v_j\}$ the condition $|N_{_{\overline{\mu(G)}}}[v_r]\cap D|=2$ implies that 
$\{v_rv_i,v_rv_j\} \subseteq E(\overline{\mu(G)})$, which equivalently means that $v_rv_i,v_rv_j \notin E(G)$.
Therefore,  $D$ consists of  two isolated vertices $v_i$ and $v_j$ of $G$.
The converse is obvious, and the proof is complete.
\end{proof}

%%%%%%%%%%%%%%%%%%%%%%%%%%%%%%%%%%%%%%%%
%%%%%%%%%%%%%%%%%%%%%%%%%%%%%%%%%%%%%%%%
%%%%%%%%%%%%%%%%%%%%%%%%%%%%%%%%%%%%%%%%
%%%%%%%%%%%%%%%%%%%%%%%%%%%%%%%%%%%%%%%%

Recall that the line graph $L(G)$ of  $G$ is the graph with vertex set $E(G)$ in which $e$ and $e'$ are adjacent in $L(G)$ if and only if the corresponding edges  share a common vertex in $G$. 
The concept of middle graph $M(G)$ of $G$ was introduced by Hamada and Yoshimura in \cite{Hamada} as an intersection graph on the vertex set of $G$, whose vertex set is $V(G)\cup E(G)$ and two vertices $a,b$ in its vertex set are adjacent whenever $a,b \in E(G)$ and $a,b$ are adjacent in $L(G)$, or $a\in V (G), b\in E(G)$ and $a,b$ are incident in $G$. 
When $V(G)=\{v_1,v_2,...,v_n\}$, then for convenient  we can set $V\big(M(G)\big)=V (G)\cup Z$, where $Z=\{z_{ij}:~ v_iv_j\in E(G)\}$ and $$E\big(M(G)\big)=\{v_iz_{ij},v_jz_{ij}:~ v_iv_j\in E(G)\}\cup E\big(L(G)\big).$$
Thus, $M(G)$ is a graph of order $|V(G)|+|E(G)|$ and size $2|E(G)|+|E(L(G))|$ and it contains the line graph $L(G)$ as an induced subgraph, see $M(C_4)$ in Figure \ref{middle}. 
%which is obtained by subdividing each edge of $G$ exactly once and then joining all the adjacent elements of $E(G)$ like done in $L(G)$.
The domination number of the middle of some famous families of graphs such as star graphs, double stars, paths, cycles, wheels, complete graphs, complete bipartite graphs and friendship graphs is considered and determined in \cite{MiddleDom}.

%%%%%%%%%%%%%%
\begin{center}
\begin{tikzpicture}
[inner sep=0.5mm, place/.style={circle,draw=black,fill=red,thick}]
%%%%%%% 

\node[circle,draw=black,fill=white,thick] (v1) at (-1,0) [label=above:$v_1$] {};
\node[circle,draw=black,fill=white,thick] (v2) at (1,0) [label=above:$v_2$] {} edge [-,thick,blue](v1) ; 
\node[circle,draw=black,fill=white,thick] (v3) at (1,-2) [label=below:$v_3$] {} edge [-,thick,blue](v2) ;
\node[circle,draw=black,fill=white,thick] (v4) at (-1,-2) [label=below:$v_4$] {} edge [-,thick,blue](v3) edge [-,thick,blue](v1) ;
\node (dots) at (0,-3.5) [label=center: $C_4$]{};

\node[circle,draw=black,fill=white,thick] (v1) at (4,0) [label=above:$v_1$] {};
\node[circle,draw=black,fill=white,thick] (v2) at (6,0) [label=above:$v_2$] {} edge [-,thick,blue](v1) ; 
\node[circle,draw=black,fill=white,thick] (v3) at (6,-2) [label=below:$v_3$] {} edge [-,thick,blue](v2) ;
\node[circle,draw=black,fill=white,thick] (v4) at (4,-2) [label=below:$v_4$] {} edge [-,thick,blue](v3) edge [-,thick,blue](v1) ;
\node[circle,draw=black,fill=white,thick] (m12) at (5,0) [label=above:$z_{_{12}}$] {};
\node[circle,draw=black,fill=white,thick] (m23) at (6,-1) [label=right:$z_{_{23}}$] {} edge [-,thick,blue](m12) ; 
\node[circle,draw=black,fill=white,thick] (m34) at (5,-2) [label=below:$z_{_{34}}$] {} edge [-,thick,blue](m23) ;
\node[circle,draw=black,fill=white,thick] (m41) at (4,-1) [label=left:$z_{_{41}}$] {} edge [-,thick,blue](m34) edge [-,thick,blue](m12) ;
\node (dots) at (5,-3.5) [label=center: $M(C_4)$]{};

\node[circle,draw=black,fill=white,thick] (v1) at (9.2,-0.2) [label=above:$v_1$] {};
\node[circle,draw=black,fill=white,thick] (v2) at (10.8,-0.2) [label=above:$v_2$] {} edge [-,thick,blue](v1) ; 
\node[circle,draw=black,fill=white,thick] (v3) at (10.8,-1.8) [label=below:$v_3$] {} edge [-,thick,blue](v1) edge [-,thick,blue](v2) ;
\node[circle,draw=black,fill=white,thick] (v4) at (9.2,-1.8) [label=below:$v_4$] {} edge [-,thick,blue](v1) edge [-,thick,blue](v2) edge [-,thick,blue](v3) ;
\node[circle,draw=black,fill=green,thick] (m12) at (10,0) [label=above:$z_{_{12}}$] {} edge [-,thick,blue](v3) edge [-,thick,blue](v4) ;
\node[circle,draw=black,fill=green,thick] (m23) at (11,-1) [label=right:$z_{_{23}}$] {} edge [-,thick,blue](v1) edge [-,thick,blue](v4) ; 
\node[circle,draw=black,fill=green,thick] (m34) at (10,-2) [label=below:$z_{_{34}}$] {} edge [-,thick,red](m12) edge [-,thick,blue](v1) edge [-,thick,blue](v2);
\node[circle,draw=black,fill=green,thick] (m41) at (9,-1) [label=left:$z_{_{41}}$] {} edge [-,thick,red](m23) edge [-,thick,blue](v2) edge [-,thick,blue](v3) ;
\node (dots) at (10,-3.5) [label=center: $\overline{M(C_4)}$]{};

%%%%%%%
\end{tikzpicture}
\captionof{figure}{    \label{middle}}
\end{center}
%%%%%%%%%%%%%%

\begin{theorem}
Let $G$ be an arbitrary graph. Then,  the middle graph $M(G)$ has not any exact doubly dominating set.
\end{theorem}
\begin{proof}
Note that if $E(G)=\emptyset$, then $E(M(G))=\emptyset$ and hence, $M(G)$ has not any exact doubly dominating set.
Thus, let $E(G)\neq \emptyset$ and assume (on the contrary) that there exists an exact doubly dominating set $D\subseteq V(M(G))$  for $M(G)$.
Choose an arbitrary vertex $v_i\in V(G)$ with $\deg_G(v_i)\geq 1$.
Since $|N_{M(G)}[v_i]\cap D|=2$ and $N_{M(G)}[v_i] \cap V(G)=\{v_i\}$, there exists at least one vertex $z_{ij}\in N_{M(G)}[v_i]\cap D$. 
Hence, we have   $v_j \in N_G(v_i)$.\\ 
Since $N_{M(G)}[v_i] \subseteq N_{M(G)}[z_{ij}]$ and 
$$\big| N_{M(G)}[m_{ij}]\cap D \big| =2= \big| N_{M(G)}[v_i]\cap D \big|,$$ 
we must have $v_j\notin D$ and $m_{js}\notin D$ for each $v_s \in N_G(v_j)\setminus\{v_i\}$. These facts imply that $N_{M(G)}[v_j]\cap D=\{m_{ij}\}$, which is a contradiction.
Therefore, $M(G)$ has not any exact doubly dominating set.
\end{proof}

\begin{theorem}
The complement of the middle graph $M(G)$ (i.e., $\overline{M(G)}$) has an exact doubly dominating set $D$ if and only if $G$ contains at least two isolated vertices and $D$ consists of two isolated vertices of $G$ or,  $G$ is (isomorphic to) the cycle $C_4$ and $D=E(C_4)$. 
\end{theorem}
\begin{proof}
For convenience, let $\mathscr{M}=M(G)$ and $\bar{\mathscr{M}}=\overline{M(G)}$.
If $G$ contains at least two isolated vertices, say $x$ and $y$, then it is easy to see that $D=\{x,y\}$ is an exact doubly dominating set for $\bar{\mathscr{M}}$.
Also, if $G=C_4$, then it is easy to check that $D=E(C_4)$  is an exact doubly dominating set for $\overline{M(C_4)}$, see Figure \ref{middle} in which the green vertices are the elements of $D$.\\
Now assume that $G$ is a graph such that $\bar{\mathscr{M}}$ has an exact doubly dominating set $D$.
Note that $V(G)\subseteq V(\mathscr{M})$ induces an independent set in $\mathscr{M}$ and hence, it induces a clique in $\bar{\mathscr{M}}$. Since  $|N_{\bar{\mathscr{M}}}[v_i]\cap D|=2$ for each $v_i\in V(G)$, we must have $|D\cap V(G)|\leq 2$. Hence, we consider the following three cases.
\\\\
{\bf Case 1.} $|D\cap V(G)| = 2$.\\
Suppose that $D\cap V(G)=\{v_i,v_j\}$. For each $v_{i'}\in V(G)$ we have $\{v_i,v_j\} \subseteq N_{\bar{\mathscr{M}}}[v_{i'}]$ and hence, 
$ N_{\bar{\mathscr{M}}}[v_{i'}] \cap D=\{v_i,v_j\}$ because $D$ is an exact doubly dominating set.
If there exists $z_{rs}\in D\cap Z$, then  the fact 
$$ N_{\bar{\mathscr{M}}}[v_i] \cap D=\{v_i,v_j\}=N_{\bar{\mathscr{M}}}[v_j] \cap D$$ 
 implies that $z_{rs}$ is non-adjacent to $v_i$ and $v_j$ in $\bar{\mathscr{M}}$ and hence, in $G$ two vertices $v_i$ and $v_j$ are incident to the edge $v_rv_s$. This means that $v_rv_s=v_iv_j$. Therefore, $D=\{v_i,v_j,z_{ij}\}$ which contradicts Theorem \ref{Matching}.
This contradiction shows that $D\cap Z=\emptyset$ and hence, $D=\{v_i,v_j\}$. If there exists an edge $v_iv_r$ in $G$ which is incident to the vertex $v_i$ in $G$, then we obtain $N_{\bar{\mathscr{M}}}[z_{ir}] \cap D\subseteq \{v_j\}$, a contradiction.
Thus,  $v_i$ (and similarly $v_j$) is an isolated vertex in $G$. This means that $D$ consists of two isolated vertices of $G$ and hence, the proof is complete in this case.
\\\\
{\bf Case 2.} $|D\cap V(G)| = 1$.\\
Assume that $D\cap V(G)=\{v_i\}$. Since $|N_{\bar{\mathscr{M}}}[v_i] \cap D|=2$, there exists a (unique) vertex $z_{rs}\in D\cap Z$ which is adjacent to $v_i$ in $\bar{\mathscr{M}}$. This implies that $v_rv_s\in E(G)$ and $v_i\notin \{v_r,v_s\}$.
Since  $|N_{\bar{\mathscr{M}}}[v_r] \cap D|=2$ and $v_i\in N_{\bar{\mathscr{M}}}[v_r]$, there exists a vertex $z_{r's'}\in D\cap Z$ which is adjacent to $v_r$ in $\bar{\mathscr{M}}$. Thus, $v_{r'}v_{s'}\in E(G)$ and $v_r\notin\{v_{r'},v_{s'}\}$. 
Note that by Theorem \ref{Matching}, $v_i$ and $z_{rs}$ are not adjacent to $z_{r's'}$ in $\bar{\mathscr{M}}$. Hence, we must have $v_i\in \{v_{r'},v_{s'}\}$ and $v_s\in \{v_{r'},v_{s'}\}$. These facts imply that $z_{r's'}=z_{is}$. Specially, we have $v_iv_s\in E(G)$.
By Theorem \ref{Matching}, $D$ induces a matching in $\bar{\mathscr{M}}$ and hence, there exists $z_{r''s''}\in D\cap Z$ which is adjacent to $z_{is}$.
Similarly, Since $v_i$ and $z_{rs}$ are not adjacent to $z_{r''s''}$ in $\bar{\mathscr{M}}$, we must have $v_i\in \{v_{r''},v_{s''}\}$ and $v_r\in \{v_{r''},v_{s''}\}$. Thus, $z_{r''s''}=z_{ir}$ and hence, $v_iv_r\in E(G)$. 
Since two edges $v_iv_r$ and $v_iv_s$  in $G$ share the common endpoint $v_i$, two vertices 
$z_{ir}$ and $z_{is}$ are not adjacent in $\bar{\mathscr{M}}$, which means $z_{is}$ is not adjacent to $z_{r''s''}$ in $\bar{\mathscr{M}}$. This is a contradiction.
Therefore, this case leads to a contradiction and hence, is impossible.
\\\\
{\bf Case 3.} $|D\cap V(G)| = 0$.\\
In this case, me must have $D\cap Z\neq\emptyset$.
Assume that $z_{ij}\in D\cap Z$. Since $z_{ij}\in Z$, we have $v_iv_j\in E(G)$. By Theorem \ref{Matching}, there exists $z_{i'j'}\in D$ which is adjacent to $z_{ij}$ in $\bar{\mathscr{M}}$. Therefore, $v_{i'}v_{j'}\in E(G)$ and $\{v_i,v_j\}\cap\{v_{i'},v_{j'}\}=\emptyset$.
Note that $z_{ij} \notin N_{\bar{\mathscr{M}}}[v_i]$ and hence, $\{z_{ij},z_{i'j'}\}$ is a proper subset of $D$.
For each $z_{i''j''}\in D\setminus\{z_{ij},z_{i'j'}\}$, Theorem \ref{Matching} implies that $z_{i''j''}$ is not adjacent to $z_{ij}$ and $z_{i'j'}$ in $\bar{\mathscr{M}}$, and hence,  we must have 
$$\{i'',j''\}\cap \{i,j\} \neq \emptyset \neq \{i'',j''\}\cap \{i',j'\}.$$
Thus, we have $D=\{z_{ij},z_{i'j'},z_{ii'},z_{jj'}\}$ or $D=\{z_{ij},z_{i'j'},z_{ij'},z_{ji'}\}$. In each of these two cases, the four elements of $D$ provides a cycle $C_4$ in $G$ and $D$ corresponds to the four edges of this $4$-cycle. Since $|D|=4$ and each vertex in $\overline{M(C_4)}$ is dominated exactly twice, $G$ can not have extra vertices nor extra edges. This means that $G=C_4$ as desired. Now the proof is complete.
\end{proof}

%%%%%%%%%%%%%%%%%%%%%%%%%%%%%%%%%%%%%%%%%

{\bf Acknowledgments:} The authors warmly thank the anonymous referees for reading this manuscript
 carefully and providing numerous valuable corrections and suggestions which improve the quality of this paper.
\\\\
The authors declare that they have no competing interests.

%%%%%%%%%%%%%%%%%%%%%%%%%%%%%%%%%%%%%%%%%%%%%%%
%%%%%%%%%%%%%%%%%%%%%%%%%%%%%%%%%%%%%%%%%%%%%%%
%%%%%%%%%%%%%%%%%%%%%%%%%%%%%%%%%%%%%%%%%%%%%%%
%%%%%%%%%%%%%%%%%%%%%%%%%%%%%%%%%%%%%%%%%%%%%%%
%%%%%%%%%%%%%%%%%%%%%%%%%%%%%%%%%%%%%%%%%%%%%%%

\end{document}